\documentclass{amsart}

\usepackage{amsmath,amssymb,amsthm}
\usepackage{multirow}

\newtheorem{theorem}{Theorem}
\newtheorem{lemma}{Lemma}
\newtheorem{corollary}{Corollary}
\newtheorem{remark}{Remark}

\begin{document}

\title[Additive uniqueness of $\mathtt{PRIMES}-1$ for multiplicative functions]{Additive uniqueness of $\mathtt{PRIMES}-1$\\ for multiplicative functions}
\author{Poo-Sung Park}
\address{Department of Mathematics Education, Kyungnam University, Changwon, Republic of Korea}
\email{pspark@kyungnam.ac.kr}

\thanks{This research was supported by Basic Science Research Program through the National Research Foundation of Korea(NRF) funded by the Ministry of  Science, ICT \& Future Planning(NRF-2017R1A2B1010761).}


\maketitle

\begin{abstract}
Let $\mathtt{PRIMES}$ be the set of all primes. We show that a multiplicative function which satisfies
\[
f(p+q-2) = f(p) + f(q) - f(2) \text{ for }p,q \in \mathtt{PRIMES}
\]
is one of the following:
\begin{enumerate}
\item $f$ is the identity function 
\item $f$ is the constant function with $f(n)=1$
\item $f(n)=0$ for $n \ge2$ unless $n$ is odd and squareful. 
\end{enumerate}
As a consequence, a multiplicative function which satisfies
\[
f(a+b) = f(a) + f(b) \text{ for }a,b \in \mathtt{PRIMES}-1
\]
is the identity function.
\end{abstract}

\section{Introduction}

Let $\mathtt{PRIMES}$ be the set of all primes. Denote the set of numbers which are one less than prime numbers by $\mathtt{PRIMES}-1$. That is,
\[
\mathtt{PRIMES}-1 = \{ 1, 2, 4, 6, 10, 12, 16, 18, 22, 28, 30, 36, 40, 42, \dots \}.
\]

In 1992, Claudia Spiro \cite{Spiro} called a set $E \subseteq \mathbb{N}$ \emph{additive uniqueness set} for the set $S$ of arithmetic functions, if there exists exactly one element $f \in S$ which satisfies
\[
f(a+b) = f(a)+f(b) \text{ for all } a, b \in E.
\]
She showed that $\mathtt{PRIMES}$ is the additive uniqueness set for the set 
\[
\{f\text{ multiplicative}, f(p_0) \ne 0\text{ for some prime }p_0\}.
\]
Many mathematicians have studied similar themes since Spiro's article. One can find various papers \cite{C-C, CFYZ, C-P, DK-K-P, D-S, Fang, I-P, Phong1997, Phong2004}. In this article, we show that $\mathtt{PRIMES}-1$ is the additive uniqueness set for the set of multiplicative functions without the condition $f(p_0)\ne0$. This shows that there exists an additive uniqueness set of the same density as $\mathtt{PRIMES}$ for multiplicative functions.

Furthermore, we characterize multiplicative functions satisfying
\[
f(p+q-2) = f(p)+f(q)-f(2)\text{ for all } p, q \in \mathtt{PRIMES}.
\]

This is also a variation of \cite{CFYZ}, which proved that if $1 \le n_0 \le 10^{6}$ and the multiplicative function $f$ satisfies
\[
f(p+q+n_0) = f(p) + f(q) + f(n_0)
\]
for $p, q \in \mathtt{PRIMES}$, then $f(n) = n$ for all $n \ge 1$ or $f(n)=0$ for all $n \ge 2$. 

\section{Results}

\begin{lemma}
If a multiplicative function $f$ satisfies
\[
f(a+b) = f(a) + f(b)
\]
for arbitrary $a, b \in \mathtt{PRIMES}-1$, then $f$ satisfies
\[
f(p+q-2) = f(p)+f(q)-f(2)
\]
for arbitrary $p, q \in \mathtt{PRIMES}$.
\end{lemma}
\begin{proof}
Note that $f(2) = f(1)+f(1) = 2$. For $p$ prime and $q=2$, we have that 
\[
f(p) = f\big( (p-1) + (2-1) \big) = f(p-1) + f(1).
\]
Thus, $f(p+q-2) = f(p-1) + f(q-1) = f(p) + f(q) - f(2)$.
\end{proof}

Now we investigate the multiplicative function satisfying $f(p+q-2) = f(p)+f(q)-f(2)$ rather than $f\big( (p-1) + (q-1) \big)$.

\begin{lemma}\label{lem:f(2)=012}
If a multiplicative function $f$ satisfies
\[
f(p+q-2) = f(p) + f(q) - f(2)
\]
for arbitrary $p, q \in \mathtt{PRIMES}$, then $f(2) = 0, 1$, or $2$. Also, $f(n)$ is determined for $3 \le n \le 18$ according to $f(2)$. But, $f(9)$ is not determined when $f(2)=0$.
\end{lemma}
\begin{proof}
Since $f$ is multiplicative, we have $f(1)=1$. From the equalities
\begin{align*}
f(4) 
&= f(3+3-2) = f(3) + f(3) - f(2) \\
f(6) 
&= f(3+5-2) = f(3) + f(5) - f(2) \\
&= f(2) f(3) \\
f(8)
&= f(3+7-2) = f(3)+f(7)-f(2) \\
&= f(5+5-2) = f(5)+f(5)-f(2) \\
f(10) 
&= f(5+7-2) = f(5) + f(7) - f(2) \\
&= f(2) f(5) \\
f(12) 
&= f(3+11-2) = f(3) + f(11) - f(2) \\
&= f(7+7-2) = f(7) + f(7) - f(2) \\
&= f(3) f(4) \\
f(20) 
&= f(11+11-2) = f(11) + f(11) - f(2) \\
&= f(4) f(5)
\end{align*}
we obtain three cases:
\[
\begin{array}{lllllll}
f(2) = 2, &f(3) = 3, &f(4) = 4, &f(5) = 5, &f(7) = 7, &f(8) = 8, &f(11) = 11; \\
f(2) = 1, &f(3) = 1, &f(4) = 1, &f(5) = 1, &f(7) = 1, &f(8) = 1, &f(11) = 1; \\
f(2) = 0, &f(3) = 0, &f(4) = 0, &f(5) = 0, &f(7) = 0, &f(8) = 0, &f(11) = 0.
\end{array}
\]

Besides, since
\begin{align*}
f(16)
&= f(5+13-2) = f(5)+f(13)-f(2) \\
&= f(7+11-2) = f(7)+f(11)-f(2) \\
f(18) 
&= f(3+17-2) = f(3)+f(17)-f(2) \\
&= f(7+13-2) = f(7)+f(13)-f(2) \\
&= f(2)f(9),
\end{align*}
we can conclude that for $3 \le n \le 18$
\[
f(n)=
\begin{cases}
n & \text{if }f(2)=2,\\
1 & \text{if }f(2)=1,\\
0 & \text{if }f(2)=0 \text{ and }n \ne 9.
\end{cases}
\]
\end{proof}

\begin{theorem}\label{thm:f(2)=2}
If a multiplicative function $f$ satisfies
\[
f(p+q-2) = f(p) + f(q) - f(2)
\]
for arbitrary $p, q \in \mathtt{PRIMES}$ and $f(2)=2$, then $f$ is the identity function.
\end{theorem}

The following corollary follows immediately.
 
\begin{corollary}
$\mathtt{PRIMES}$ is an additive uniqueness set for the set of multiplicative functions.
\end{corollary}

To prove the theorem we need some lemmas. In the first step, we use the fact that every even positive integer smaller than some upper bound is expressible as the sum of two primes. In the second step, we apply the theorem that almost every even number is the sum of two primes.

\begin{lemma}\label{lem:SumOfTwoSquaresWhenf(2)=2}
Let $f$ satisfy the hypothesis of Theorem \ref{thm:f(2)=2}. If every even number $2m$, with $4 \le 2m \le 2N$, can be written as the sum of two primes, then we have $f(n)=n$ for all $n \le N-2$.
\end{lemma}
\begin{proof}
The proof is almost identical to \cite[p.237 Proof of Lemma 4]{Spiro}. For convenience we write the proof here.

By the Lemma \ref{lem:f(2)=012}, we have that $f(n)=n$ for $1 \le n \le 18$. Assume that $M$ is an integer with $18 \le M \le N-3$ and that we have $f(n)=n$ for all $n \le M$. We will show that $f(M+1)=M+1$.

If $M+1$ is even, then $M+3$ is sum of two primes, say $p$ and $q$. Thus, 
\[
f(M+1) = f( p+q-2 ) = f(p)+f(q)-f(2) = M+1.
\]

Now, suppose that $M+1$ is odd. If $M+1$ is prime, then let $q \in \{3,5\}$ be chosen so that $M+1+q \equiv 0 \pmod{4}$. Note that
\[
f(M+1+q-2) = f(M+1) + f(q) - f(2) 
\]
and thus
\[
f(2)f\!\left( \frac{M+1+q-2}{2} \right) = f(M+1) + f(q) - f(2)
\]
by the multiplicity of $f$. Since $f(n)=n$ for $n \le M$,
\[
2 \!\left( \frac{M+1+q-2}{2} \right) = f(M+1) + q - 2
\]
from which it follows that $f(M+1)=M+1$.

If $M+1 = ab$ with relatively primes $a < M$ and $b < M$, then $f(M+1) = f(ab) = f(a)f(b) = M+1$. Thus, it remains to show $f(M+1)=M+1$ when $M+1$ is a power of an odd prime.

Asume that there exist two primes $p$ and $q$ satifying $2(M+1)+2 = p+q$ with $p < (M+1)+1 < q$.
\[
f\big( 2(M+1) \big) = 2f(M+1) = f(p)+f(q)-f(2) = p+f(q)-2
\]
Choose a prime $r \in \{3, 5, 7, 17\}$ such that $q+r \equiv 6 \pmod{8}$. Then,
\[
f(q + r - 2) = f(q) + f(r) - f(2) = f(q) + r - 2
\]
and
\[
f(q+r-2) = f(4)f\!\left( \frac{q+r-2}{4} \right) = 4 \cdot \frac{q+r-2}{4}
\]
from $\frac{q+r-2}{4} < M$.

Hence, $f(M+1) = M+1$. The lemma is proved.
\end{proof}

If the famous Goldbach Conjecture could be proved true, Theorem \ref{thm:f(2)=2} would be true. But, Goldbach Conjecture is not yet proved. It was numerically verified up to $4\times10^{18}$ \cite{S-H-P}.

To bypass the Goldbach Conjecture, we construct a specific set $H$ and use the fact that the set of even integers which are not the sum of two primes has density zero. In the following lemma, $v_p(n)$ means the exponent of $p$ in the prime factorization of $n$ when $p$ is a prime and $n$ is a positive integer. The set $H$ was defined by Spiro and the numerical verification of Goldbach Conjecture was up to $2\times10^{10}$ at that time. We would call the set $H$ in the lemma the \emph{Spiro set}.

\begin{lemma}\label{lem:m+q_in_H}
Let
\[
H = \{ n \,|\, v_p(n) \le 1 \text{ if } p>1000; v_p(n) \le \lfloor 9\log_p{10} \rfloor - 1 \text{ if } p<1000 \}.
\]
For any integer $m > 10^{10}$, there is an odd prime $q \le m-1$ such that $m+q \in H$. 
\end{lemma}

\begin{proof}
This lemma is the consequence of \cite[Lemma 2.4]{CFYZ} which follows the proof of \cite[Lemma 5]{Spiro}.
\end{proof}

The Spiro set $H$ is composed of the following prime powers and their products:
\[
\begin{array}{rrrrrrrr}
2, & 2^{2}, & \dotsc, & 2^{29}, 
& 3, & 3^{2}, & \dotsc, & 3^{18}, \\
5, & 5^{2}, & \dotsc, & 5^{12}, 
& 7, & 7^{2}, & \dotsc, & 7^{10}, \\
11, & 11^{2}, & \dotsc, & 11^{8}, 
& 13, & 13^{2}, & \dotsc, & 13^{8}, \\
17, & 17^{2}, & \dotsc, & 17^{7}, 
& 19, & 19^{2}, & \dotsc, & 19^{7}, \\
p, & p^2, & \dotsc, & p^6~ & \multicolumn{4}{l}{~~\text{ for }p=23, 29, 31,}\\
q, & q^2, & \dotsc, & q^5~ & \multicolumn{4}{l}{~~\text{ for }q=37, 41, \dotsc, 61,}\\
r, & r^2, & r^3, & r^4~ & \multicolumn{4}{l}{~~\text{ for }r=67, 71, \dotsc, 173,}\\
s, & s^2, & s^3~ & ~ & \multicolumn{4}{l}{~~\text{ for }s=179, 181, \dotsc, 997,}\\
1009, & 1013, & 1019, & 1021, & 1031, & 1033, & 1039, & \dotsc.
\end{array}
\]
The smallest positive integer which is not in $H$ is $1009^2=1018081$. Clearly, every divisor of an element of $H$ is also in $H$.

\begin{lemma}[\protect{\cite[Lemma 7]{Spiro}}]\label{lem:H_n}
For any positive integer $n$, put
\[
H_n = 
\begin{cases}
\{ mn \,:\, m \in H, (m,n)=1\} & \text{if } 2 \mid n; \\
\{ 2mn \,:\, 2m \in H, (m,n)=1\} & \text{if } 2 \nmid n. \\
\end{cases}
\]
Then $H_n$ satisfies the following properties:
\begin{enumerate}
\item Every element of $H_n$ is even.
\item The set $H_n$ has positive lower density.
\end{enumerate}
\end{lemma}

\begin{lemma}[\protect{\cite[Lemma 6]{Spiro}}]\label{lem:almost_every_integer}
Almost every even positive integer is expressible as the sum of two primes.
\end{lemma}

\begin{lemma}\label{lem:n_in_H}
Provided $f(2)=2$, then $f(n) = n$ for all $n \in H$. 
\end{lemma}

\begin{proof}
If $n < 10^{10}$, then $f(n)=n$ from Lemma \ref{lem:SumOfTwoSquaresWhenf(2)=2} and the numerical verification of Goldbach Conjecture. Let $n \in H$ with $n > 10^{10}$ and assume that $f(m)=m$ for all $m \in H$ with $m < n$. If $n$ is not a prime power, then $f(n) = f(a)f(b)$ with $(a,b)=1$ and $a,b > 1$. Since $f(a)=a$ and $f(b)=b$ by the induction hypothesis, $f(n)=n$. 

Now, if $n$ is a prime power, then $n$ is a prime by the definition of $H$. We have that there exists an odd prime $q < n-2$ with $(n-2) + q \in H$ by Lemma \ref{lem:m+q_in_H}. 
Since $n$ is odd, $n+q-2$ is even and thus $n+q-2 = 2^s k$ with $1 \le s \le 29$ and $k$ odd. Then, $f(2^s)=2^s$ and $f(k)=k$ by the induction hypothesis since $2^s < 10^{10} < n$ and $k < n$. Since $n$ is prime, $f(n) = f(n+q-2) - f(q) + 2 = (n+q-2) - q + 2 = n$.
\end{proof}

\begin{proof}[Proof of Theorem \ref{thm:f(2)=2}]
If the theorem is false, let $n$ be the minimal counterexample. Clearly, $n > 10^{10}$. Consider $k \in H_n$. Then, $n \mid k$, $(k/n,n)=1$, and $k/n \in H$.

We can deduce that
\[
f(k) = f(n) f\!\left( \frac{k}{n} \right) = f(n) \, \frac{k}{n}
\]
for all $k \in H_n$, since $f(k/n)=k/n$ by Lemma \ref{lem:n_in_H}.

Note that $H$ contains all primes. Thus, $f(p)=p$ for $p \in \mathtt{PRIMES}$. If $k = p+q-2$  for some $p, q \in \mathtt{PRIMES}$, then $f(k) = f(p)+f(q)-f(2) = k$ and thus $f(n) = n$. But, since we choose $n$ to be the counterexample, no element $k \in H_n$ can be of the form $p+q-2$. That is, $k+2$ cannot be expressible as the sum of two primes, which contradicts Lemma \ref{lem:almost_every_integer} by Lemma \ref{lem:H_n}.

Ergo, $f(n)=n$ for all positive integers $n$.

\end{proof}

\begin{theorem}\label{thm:f(2)=1}
If a multiplicative function $f$ satisfies
\[
f(p+q-2) = f(p) + f(q) - f(2)
\]
for arbitrary $p, q \in \mathtt{PRIMES}$ and $f(2)=1$, then $f(n) = 1$ for all positive integers $n$.
\end{theorem}
\begin{proof}
We can prove in the way similar to the proof of Theorem \ref{thm:f(2)=2} under the condition $f(2)=1$.
\end{proof}

\begin{theorem}\label{thm:f(2)=0}
If a multiplicative function $f$ satisfies
\[
f(p+q-2) = f(p) + f(q) - f(2)
\]
for arbitrary $p, q \in \mathtt{PRIMES}$ and $f(2)=0$, then $f(n)$ vanishes when $n$ is even or nonsquareful. Besides, $f(p^s)$ can be assigned to be an arbitrary number for odd $p \in \mathtt{PRIMES}$ and $s \ge 2$.
\end{theorem}
\begin{proof}
We have that $f(p) = 0$ for $p$ prime $\le 18$ by Lemma \ref{lem:f(2)=012}. Let $p$ be a prime $\ge 19$. We choose $q \in \{ 3, 5 \}$ such that $p+q \equiv 0 \pmod{4}$. Then
\begin{align*}
f(p+q-2) 
&= f(p)+f(q)-f(2) = f(p) \\
&= f(2)f\!\left( \frac{p+q-2}{2} \right) = 0.
\end{align*}
So, $f(p)=0$ and thus $f(n)=0$ for every nonsquareful integer $n$.

We can reason that $f(2n) = 0$ in the way similar to the proof of Theorem \ref{thm:f(2)=2}. Thus, $f(2^s) = 0$ for every positive integer $s$.

Note that $f(2^s k) = f(2^s)f(k) = 0$ gives no information about $f(k)$ when $k$ is odd. Also, if $n = p+q-2$ is odd, then it is possible only when $n$ itself is a prime and $q=2$. Thus we cannot determine $f(p^s)$ for $p$ prime and $s \ge 2$.
\end{proof}

\begin{remark}
If $f$ is completely multiplicative, then, clearly, $f(n)=0$ for all $n \ge 2$. If not, we can freely construct a multiplicative function $f$.

For example, let a multiplicative function $f$ be defined as
\[
f(n) = \begin{cases}
1 & \text{if $n$ is odd and squareful},\\
0 & \text{otherwise}.
\end{cases}
\]
Then, $f$ satisfies the condition $f(p+q-2) = f(p) + f(q) - f(2)$ for all $p, q \in \mathtt{PRIMES}$.
\end{remark}

\section*{Acknowledgment}

The author would like to thank East Carolina University for its support and hospitality.

\end{document}